\documentclass[11pt]{amsart}
\usepackage{a4wide}
\usepackage{amssymb}
\usepackage{mathrsfs}
\theoremstyle{plain}
\newtheorem{theorem}{Theorem}
\newtheorem*{theorem*}{Theorem}
\newtheorem{lemma}[theorem]{Lemma}
\newtheorem{proposition}[theorem]{Proposition}
\newcommand{\NN}{\mathbf{N}}
\newcommand{\ZZ}{\mathbf{Z}}

\begin{document}

\title[The topological full group of a minimal Cantor {$\ZZ^2$}-system]{On the topological full group\\ of a minimal Cantor {$\ZZ^2$}-system}
\author[G. Elek and N. Monod]{G\' abor Elek and Nicolas Monod}
\address{EPFL, 1015 Lausanne, Switzerland}
\thanks{Work supported in part by a Marie Curie grant, the European Research Council and the Swiss National Science Foundation.}
%
\begin{abstract}
Grigorchuk and Medynets recently announced that the topological full group of a
minimal Cantor $\ZZ$-action is amenable. They asked whether the statement holds
for all minimal Cantor actions of general amenable groups as well. We answer in the negative by producing
a minimal Cantor $\ZZ^2$-action for which the topological full group contains a non-abelian free group.
\end{abstract}
\maketitle

\section{Introduction}
Let $G$ be a group acting on a compact space $\Sigma$ by homeomorphisms. The \emph{topological full group} associated to this action is the group of all homeomorphisms of $\Sigma$ that are piecewise given by elements of $G$, each piece being open. Thus there are finitely many pieces at a time, all are clopen, and this construction is most interesting when $\Sigma$ is a Cantor space. The importance of the topological full group has come to the fore in the classification results of Giordano--Putnam--Skau~\cite{Giordano-Putnam-Skau95,Giordano-Putnam-Skau99}.

\medskip
Grigorchuk and Medynets announced that the topological full group of a minimal Cantor $\ZZ$-action is amenable~\cite{Grigorchuk-Medynets}. This is particularly interesting in combination with the work of Matui~\cite{Matui06}, who showed that the derived subgroup is often a finitely generated simple group. Grigorchuk--Medynets further asked in~\cite{Grigorchuk-Medynets} whether their result holds for actions of general amenable groups as well. We shall prove that it fails already for the group~$\ZZ^2$:

\begin{theorem}\label{main}
There exists a free minimal Cantor $\ZZ^2$-action whose topological full group contains a non-abelian free group.
\end{theorem}

Three comments are in order, see the end of this note:

\medskip\noindent
\textbf{1.} There also exist free minimal Cantor $\ZZ^2$-actions whose topological full group is amenable, indeed locally virtually abelian.

\smallskip\noindent
\textbf{2.} Minimality is fundamental for the study of topological full groups. Even for~$\ZZ$, it is easy to construct Cantor systems whose topological full group contains a non-abelian free group (using e.g.\ ideas from~\cite{vanDouwen} or~\cite{Glasner-Monod}).

\smallskip\noindent
\textbf{3.} Our example will be a minimal subshift and in this situation the topological full group is sofic by a result of~\cite{Elek-Szabo05}.

\section{Proof of the Theorem}
We realize the Cantor space as the space $\Sigma$ of all proper edge-colourings of the ``quadrille paper'' two-dimensional Euclidean lattice by the letters $A,B,C,D,E,F$ (with the topology of pointwise convergence relative to the discrete topology on the finite set of letters). Recall here that an edge-colouring is called \emph{proper} if the edges adjacent to a given vertex are coloured differently. There is a natural $\ZZ^2$-action on $\Sigma$ by homeomorphisms defined by translations.

\medskip
To each letter $x\in\{A, \ldots, F\}$ corresponds a continuous involution of $\Sigma$, which we still denote by the same letter. It is defined as follows on $\sigma\in\Sigma$: if the vertex zero is connected to one of its four neighbours $v$ by an edge labelled by $x$, then $v$ is uniquely determined and $x\sigma$ will be the colouring $\sigma$ translated towards $v$ (i.e.\ the origin is now where $v$ was). Otherwise, $x\sigma=\sigma$. This involution is contained in the topological full group of the $\ZZ^2$-action.

\medskip
We have thus a homomorphism from the free product $\langle A\rangle * \cdots * \langle F\rangle$ to the topological full group. Notice that this free product preserves any $\ZZ^2$-invariant subset of $\Sigma$. We shall establish Theorem~\ref{main} by proving that $\Sigma$ contains a minimal non-empty closed $\ZZ^2$-invariant subset $M$ on which the $\ZZ^2$-action is free and on which the action of $\Delta:=\langle A\rangle * \langle B\rangle * \langle C\rangle$ is faithful. This implies the theorem indeed, for $\Delta$ has a (finite index) non-abelian free subgroup.

\bigskip
A \emph{pattern} of a colouring $\sigma\in \Sigma$ is the isomorphism class of a finite labelled subgraph of $\sigma$. We call $\sigma$ \emph{homogeneous} if for any pattern
$P$ of $\sigma$ there is a number $f(P)$ such that the $f(P)$-neighbourhood of any vertex in the lattice contains the pattern $P$. 
The following facts are well-known and elementary (see e.g.~\cite{Gottschalk46}).

\begin{lemma}\label{lem:Got}
The orbit closure of $\sigma\in \Sigma$ is minimal if and only if $\sigma$ is homogeneous. In that case, any $\tau$ in the orbit closure has the same patterns as $\sigma$ and is homogenous with the same function $f$.\qed
\end{lemma}

Now, we first enumerate the non-trivial elements of the free product $\Delta$.
Then, we label the integers with the natural numbers in such a way that the following property holds: for each $i\in \NN$ there is $g(i)\geq 1$ such that any subinterval of length $g(i)$ in $\ZZ$ contains at least one element labelled by $i$. Such a labelling exists: for instance, label an integer by the exponent of~$2$ in its prime factorization (with an arbitrary adjustment for~$0$).

\medskip
We use the labelling above to construct a specific proper edge-colouring $\lambda\in\Sigma$. Let $w$ be a word in $\Delta$ that is the $i$-th in the enumeration. Consider the vertical vertex-lines $(v,\cdot)$ in the lattice such that $v$ is labelled by $i$. Colour those vertical lines the following way. Starting at the point $(v, 0)$, copy the string $w$ onto the half-line above, beginning from the right end of $w$ (i.e.\ write $w^{-1}$ upwards).  Then colour the following edge by $D$, then copy the string $w$ again and repeat the process ad infinitum. Also, continue the process below $(v, 0)$ so as to obtain a periodic colouring of the whole vertical line. Repeating the process for all non-trivial words $w$, we have coloured all vertical lines. Finally, colour all horizontal lines periodically with $E$ and $F$.

\medskip
The resulting colouring $\lambda$ has the following property.
For any non-trivial $w\in \Delta$ there is a number $h(w)$ such that the
$h(w)$-neighbourhood of any vertex of the lattice contains a vertical string
of the form $w^{-1}D$. Let $\Omega(\lambda)\subseteq \Sigma$ be the $\ZZ^2$-orbit closure of
$\lambda$. Then all the elements of $\Omega(\lambda)$ have the same property.
Now, let $M$ be an arbitrary minimal subsystem of $\Omega(\lambda)$ (in fact it is easy to see that $\lambda$ is homogeneous and hence $\Omega(\lambda)$ is already minimal). Notice that the $\ZZ^2$-action on $M$ is free because $\lambda$ has no period.
In order to prove the theorem, it is enough to show that for any $\sigma$ in
$M$ and any non-trivial $w\in \Delta$ there exists a $\ZZ^2$-translate of $\sigma$ which is not
fixed by $w$.

\medskip
Pick thus any $\sigma\in M$. Then, by the above property of the
orbit closure, there exists a translate $\tau$ of $\sigma$ such that the vertical
half-line pointing upwards from the origin starts with the string $w^{-1}D$. Hence
if we apply $w$ to the translate we reach a point $\tau$ such that
the colour of the edge pointing upwards from the the origin is coloured by $D$.
Thus $\tau$ is not fixed by $w$, finishing the proof.\qed

\section{Comments}
Some $\ZZ^2$-systems have a completely opposite behaviour to the ones constructed for Theorem~\ref{main}. We shall see this by extending the method of Proposition~2.1 in~\cite{Matui_fullII}.

\medskip
Recall that the \emph{$p$-adic odometer} is the minimal Cantor system given by adding~$1$ in the ring $\ZZ_p$ of $p$-adic integers. Taking the direct product, we obtain a minimal Cantor $\ZZ^2$-action on $\Sigma:= \ZZ_p\times \ZZ_p$.  The proposition below and its proof can be immediately extended to products of more general odometers.

\begin{proposition}\label{odo}
The full group of this minimal Cantor $\ZZ^2$-system is an increasing union of virtually abelian groups.
\end{proposition}

\begin{proof}[Proof (compare~\cite{Matui_fullII})]
Consider $\ZZ_p$ as the space of $\ZZ/p\ZZ$-valued (infinite) sequences. Given a pair of finite sequences of length~$n$, we obtain an $n$-cylinder set in $\Sigma$ as the space of pairs of sequences starting with the given prefixes. Thus, $n$-cylinders determine a partition $\mathscr{P}_n$ of $\Sigma$ into $p^{2n}$ clopen subsets. Moreover, the clopen partition associated to any given element $g$ of the topological full group can be refined to $\mathscr{P}_n$ when $n$ is large enough. It remains only to observe that the collection of all such $g$, when $n$ is fixed, is a subgroup of the semi-direct product $(\ZZ^2)^{\mathscr{P}_n} \rtimes \mathrm{Sym}(\mathscr{P}_n)$, where $\mathrm{Sym}(\mathscr{P}_n)$ is the permutation group of the co\"ordinates indexed by $\mathscr{P}_n$.
\end{proof}

Regarding the second comment of the introduction, suffice it to say that a \emph{generic} proper colouring of the linear graph by three letters $A,B,C$ gives a faithful non-minimal representation of the free product $\langle A\rangle * \langle B\rangle * \langle C\rangle$ into the topological full group of the associated $\ZZ$-subshift (compare~\cite{vanDouwen} or~\cite{Glasner-Monod} for generic constructions).

\bigskip
As for the last comment, Proposition~5.1(1) in~\cite{Elek-Szabo05} implies that the topological full group of any minimal subshift of any amenable group is a sofic group (in the notations of~\cite{Elek-Szabo05}, the kernel $N_\Gamma$ is trivial by an application of Lemma~\ref{lem:Got}). In combination with Matui's results~\cite{Matui06}, this already shows the existence of a sofic finitely generated infinite simple group without appealing to~\cite{Grigorchuk-Medynets}.

\bibliographystyle{../BIB/abbrv}
\bibliography{../BIB/ma_bib}

\end{document}